\documentclass[english,oneside,a4paper,12pt]{article}
\usepackage{amssymb}
\usepackage{latexsym}
\usepackage[english]{babel}
\usepackage{amsmath}
\usepackage{amsfonts}
\usepackage{amsbsy}
\usepackage[mathscr]{eucal}
\usepackage[latin1]{inputenc}
\usepackage{verbatim}
\usepackage{mathrsfs}
\usepackage{graphicx}
\usepackage{float}
\usepackage{bbm}
\usepackage{lscape}
\usepackage{mathdots}
\usepackage{empheq}
\usepackage[dvipsnames]{xcolor}
\usepackage{array}
\usepackage{a4wide}
\usepackage{enumitem}
\usepackage[dvips,colorlinks,bookmarks,breaklinks]{hyperref}  
\hypersetup{linkcolor=black,citecolor=black,filecolor=black,urlcolor=black}
\usepackage{natbib}

\newtheorem{theorem}{Theorem}
\newtheorem{corollary}{Corollary}

\newenvironment{proof}
{\begin{trivlist}\item[\hskip%
\labelsep{{\it \noindent Proof.}}]}{\hfill $\square$
\end{trivlist}}

\newcounter{counter}
\newcommand{\counter}{\stepcounter{counter}\thecounter}

\newenvironment{example}
{\begin{trivlist}\item[\hskip%
\labelsep{{\sc \noindent Example.}}]}{\hfill
\end{trivlist}}

\newenvironment{acknowledgements}
{\begin{trivlist}\item[\hskip%
\labelsep{\bf \noindent Acknowledgements.}]}{\hfill
\end{trivlist}}

\numberwithin{equation}{section}

\allowdisplaybreaks

\begin{document}
\begin{center}
{\LARGE \textbf{\!\!\mbox{Strong laws of large numbers for pairwise PQD} \\ random variables: a necessary condition}} \\[25pt]
{\Large Jo\~{a}o Lita da Silva\footnote{\textit{E-mail address:} \texttt{jfls@fct.unl.pt; joao.lita@gmail.com}}} \\
\vspace{0.1cm}
\textit{Department of Mathematics and GeoBioTec \\ Faculty of Sciences and Technology \\
NOVA University of Lisbon \\ Quinta da Torre, 2829-516 Caparica,
Portugal}
\end{center}


\bigskip

\begin{abstract}
A necessary condition is given for a sequence of identically distributed and pairwise positively quadrant dependent random variables obeying the strong laws of large numbers with respect to the normalising constants $n^{1/p}$ $(1 \leqslant p < 2)$.
\end{abstract}

\bigskip

{\textit{Key words and phrases:} Strong law of large numbers, positively quadrant dependent random variables, (positively) associated random variables}

\bigskip

{\small{\textit{2010 Mathematics Subject Classification:} 60F15}}

\bigskip

\section{A necessary condition}

\indent

Although the idea of positive dependence traces back to the forties by way of particular examples and without explicit referral to any dependent structure (see \citealt{Hoeffding48}; \citealt{Lehmann51} or \citealt{Blum61}), formal definitions only emerged in \citealt{Esary67} and \citealt{Lehmann66}. Throughout, we shall focus on the concepts of (positive) association and positive quadrant dependence, whose definitions can be found in the previous two references.

A finite sequence of random variables $\{X_{1},\ldots, X_{n} \}$ is said to be \emph{(positively) associated} if for coordinatewise nondecreasing functions $f$ and $g$ defined on $\mathbf{R}^{n}$,
\begin{equation*}
\mathrm{Cov} \big(f (X_{1},\ldots, X_{n}), g(X_{1},\ldots, X_{n}) \big) \geqslant 0
\end{equation*}
whenever the covariance is defined. A sequence $\{X_{n}, n \geqslant 1\}$ is said to be (positively) associated if every finite subfamily is (positively) associated. A sequence $\{X_{n}, \, n \geqslant 1 \}$ of random variables is said to be \emph{pairwise positively quadrant dependent} (pairwise PQD, in short) if
\begin{equation*}
\mathbb{P} \left\{X_{k} \leqslant x_{k}, X_{j} \leqslant x_{j}  \right\} - \mathbb{P} \left\{X_{k} \leqslant x_{k} \right\} \mathbb{P} \left\{X_{j} \leqslant x_{j}  \right\} \geqslant 0
\end{equation*}
for all reals $x_{k}, x_{j}$ and all positive integers $k,j$ such that $k \neq j$. Concerning elementary properties involving the aforementioned concepts, we refer the reader to \citealt{Esary67} and \citealt{Lehmann66}.

Associated to a probability space $(\Omega, \mathcal{F}, \mathbb{P})$, we shall consider the space $\mathscr{L}_{p}$ $(p>0)$ of all measurable functions $X$ (necessarily random variables) for which $\mathbb{E} \lvert X \rvert^{p} < \infty$. Given an event $A$ we shall denote the indicator random variable of the event $A$ by $I_{A}$. 

Over the last decades, some authors have established strong laws of large numbers for positively dependent random variables (see, for instance, \citealt{Birkel89}; \citealt{Chen19}; \citealt{Louhichi00} or \citealt{Lita18a}). Recently, Chen and Sung (\citealt{Chen19}) considered a sequence $\{X_{n}, \, n \geqslant 1 \}$ of (positively) associated random variables stochastically dominated by a random variable $X \in \mathscr{L}_{p}$, $1 \leqslant p < 2$ (i.e. for which there exists a constant $C>0$ such that $\sup_{n} \mathbb{P}\{\lvert X_{n} \rvert > t \} \leqslant C \mathbb{P} \{\lvert X \rvert > t \}$ for each $t>0$ and some random variable $X \in \mathscr{L}_{p}$, $1 \leqslant p < 2$), proving that if the following covariance condition
\begin{equation}\label{eq:1.1}
\sum_{1 \leqslant k < j \leqslant \infty} j^{-2/p} G_{X_{k},X_{j}} \big(k^{1/p},j^{1/p} \big) < \infty
\end{equation}
holds, where $G_{X_{k},X_{j}}(u,v) := \int_{-v}^{v} \int_{-u}^{u} \big[\mathbb{P} \{X_{k} > x, X_{j} > y \} - \mathbb{P} \{X_{k} > x \} \mathbb{P} \{X_{j} > y \}\big] \, \mathrm{d}x \mathrm{d}y$, then $\sum_{k=1}^{n} (X_{k} - \mathbb{E} \, X_{k})/n^{1/p} \overset{\textnormal{a.s.}}{\longrightarrow} 0$.

On the other hand, from Remark 1 of \citealt{Lita18b}, we can state that if $\{X_{n}, \, n \geqslant 1 \}$ is a sequence of pairwise PQD random variables stochastically dominated by a random variable $X \in \mathscr{L}_{1}$ satisfying $\sum_{1 \leqslant k < j \leqslant \infty} j^{-2} G_{X_{k},X_{j}}(k,j) < \infty$ then $\sum_{k=1}^{n} (X_{k} - \mathbb{E} \, X_{k})/n \overset{\textnormal{a.s.}}{\longrightarrow} 0$ (see also Theorem 3 of \citealt{Chen19}).

Our purpose in this short note is to provide a necessary condition for a sequence $\{X_{n}, \, n \geqslant 1 \}$ of identically distributed and pairwise PQD random variables to satisfy $\sum_{k=1}^{n} (X_{k} - \mathbb{E} \, X_{k})/n^{1/p} \overset{\textnormal{a.s.}}{\longrightarrow} 0$. During the whole of this paper, we shall define $\Delta_{k,j}(x,y) := \mathbb{P} \{X_{k} > x, X_{j} > y \} - \mathbb{P} \{X_{k} > x \} \mathbb{P} \{X_{j} > y \}$.

\begin{theorem}\label{thr:1}
Let $1 \leqslant p < 2$ and $\{X_{n}, \, n \geqslant 1 \}$ be a sequence of identically distributed pairwise PQD random variables. If $\sum_{k=1}^{n} (X_{k} - c)/n^{1/p} \overset{\textnormal{a.s.}}{\longrightarrow} 0$ for some finite constant $c$, and
\begin{equation}\label{eq:1.2}
\sum_{1 \leqslant k < j \leqslant \infty} (kj)^{-1/p} G_{X_{k},X_{j}} \big(k^{1/p},j^{1/p} \big) < \infty
\end{equation}
then $\sum_{k=1}^{\infty} \mathbb{P} \left\{\lvert X_{1} \rvert > k^{1/p} \right\} < \infty$.
\end{theorem}

\begin{proof}
Firstly, if $\sum_{k=1}^{n} (X_{k} - c)/n^{1/p} \overset{\textnormal{a.s.}}{\longrightarrow} 0$ then
\begin{equation}\label{eq:1.3}
\frac{X_{n}}{n^{1/p}} = \frac{X_{n} - c}{n^{1/p}} + \frac{c}{n^{1/p}} = \frac{\sum_{k=1}^{n}(X_{k} - c)}{n^{1/p}} - \frac{(n - 1)^{1/p}}{n^{1/p}} \frac{\sum_{k=1}^{n-1} (X_{k} - c)}{(n - 1)^{1/p}} + \frac{c}{n^{1/p}} \overset{\textnormal{a.s.}}{\longrightarrow} 0.
\end{equation}
Arguing by contradiction, let us assume that $\sum_{k=1}^{\infty} \mathbb{P} \left\{X_{1} > k^{1/p} \right\} = \infty$. Supposing $A_{k} := \left\{X_{k} > k^{1/p} \right\}$ and fixing an $\varepsilon > 1$, we have for every $k \neq j$,
\begin{align*}
\int_{j^{1/p}/\varepsilon}^{j^{1/p}} \int_{k^{1/p}/\varepsilon}^{k^{1/p}} \mathbb{P} \left\{X_{k} > x, X_{j} > y \right\} \mathrm{d}x \mathrm{d}y &\geqslant \int_{j^{1/p}/\varepsilon}^{j^{1/p}} \int_{k^{1/p}/\varepsilon}^{k^{1/p}} \mathbb{P} (A_{k} \cap A_{j}) \, \mathrm{d}x \mathrm{d}y \\
&= k^{1/p} j^{1/p} \left(\frac{\varepsilon - 1}{\varepsilon} \right)^{2} \mathbb{P} (A_{k} \cap A_{j})
\end{align*}
and
\begin{align*}
&\mathbb{P} (A_{k} \cap A_{j}) \\
& \; \; \leqslant \left(\frac{\varepsilon}{\varepsilon - 1} \right)^{2} \frac{1}{k^{1/p} j^{1/p}} \int_{j^{1/p}/\varepsilon}^{j^{1/p}} \int_{k^{1/p}/\varepsilon}^{k^{1/p}} \mathbb{P} \left\{X_{k} > x, X_{j} > y \right\} \mathrm{d}x \mathrm{d}y \\
&\; \; = \left(\frac{\varepsilon}{\varepsilon - 1} \right)^{2} \frac{1}{k^{1/p} j^{1/p}} \int_{j^{1/p}/\varepsilon}^{j^{1/p}} \int_{k^{1/p}/\varepsilon}^{k^{1/p}} \left[\Delta_{X_{k},X_{j}}(x,y) + \mathbb{P} \left\{X_{k} > x \right\} \mathbb{P} \left\{ X_{j} > y \right\} \right] \mathrm{d}x \mathrm{d}y \\
&\; \; \leqslant \left(\frac{\varepsilon}{\varepsilon - 1} \right)^{2}  \frac{1}{k^{1/p} j^{1/p}} \int_{j^{1/p}/\varepsilon}^{j^{1/p}} \int_{k^{1/p}/\varepsilon}^{k^{1/p}}  \Delta_{X_{k},X_{j}}(x,y) \, \mathrm{d}x \mathrm{d}y + \mathbb{P} \left\{X_{1} > \frac{k^{1/p}}{\varepsilon} \right\} \mathbb{P} \left\{X_{1} > \frac{j^{1/p}}{\varepsilon} \right\}.
\end{align*}
Thus,
\begin{align*}
\frac{\sum_{k,j=1}^{n} \mathbb{P} (A_{k} \cap A_{j})}{\left[\sum_{k=1}^{n} \mathbb{P}(A_{k}) \right]^{2}} &\leqslant 2 \left(\frac{\varepsilon}{\varepsilon - 1} \right)^{2} \cdot \frac{\sum_{1 \leqslant k < j \leqslant n} (kj)^{-1/p} \int_{j^{1/p}/\varepsilon}^{j^{1/p}} \int_{k^{1/p}/\varepsilon}^{k^{1/p}} \Delta_{X_{k},X_{j}}(x,y) \, \mathrm{d}x \mathrm{d}y}{\left(\sum_{k=1}^{n} \mathbb{P}\{X_{1} > k^{1/p} \} \right)^{2}} \\
& \quad + \frac{1}{\sum_{k=1}^{n} \mathbb{P}\{X_{1} > k^{1/p} \}} + \left(\frac{\sum_{k=1}^{n} \mathbb{P} \left\{\varepsilon X_{1} > k^{1/p} \right\}}{\sum_{k=1}^{n} \mathbb{P}\{X_{1} > k^{1/p} \} }\right)^{2}.
\end{align*}
From
\begin{align*}
1 \leqslant & \frac{\sum_{k=1}^{n} \mathbb{P} \left\{\varepsilon X_{1} > k^{1/p} \right\}}{\sum_{k=1}^{n} \mathbb{P}\{X_{1} > k^{1/p} \}} \leqslant \frac{\int_{0}^{n} \mathbb{P} \left\{\varepsilon X_{1} > t^{1/p} \right\} \mathrm{d}t}{\sum_{k=1}^{n} \mathbb{P}\{X_{1} > k^{1/p} \} } = \frac{\varepsilon^{p} \int_{0}^{n/\varepsilon^{p}} \mathbb{P} \left\{X_{1} > s^{1/p} \right\} \mathrm{d}s}{\sum_{k=1}^{n} \mathbb{P}\{X_{1} > k^{1/p} \} } \\
&\leqslant \frac{\varepsilon^{p} \sum_{k=0}^{\left\lceil \frac{n}{\varepsilon^{p}} \right\rceil - 1} \mathbb{P} \left\{X_{1} > k^{1/p} \right\}}{\sum_{k=1}^{n} \mathbb{P}\{X_{1} > k^{1/p} \} } \leqslant \frac{\varepsilon^{p} \sum_{k=0}^{n} \mathbb{P} \left\{X_{1} > k^{1/p} \right\}}{\sum_{k=1}^{n} \mathbb{P}\{X_{1} > k^{1/p} \} } \leqslant \frac{\varepsilon^{p}}{\sum_{k=1}^{n} \mathbb{P}\{X_{1} > k^{1/p} \}} + \varepsilon^{p},
\end{align*}
where $\lceil u \rceil$ denotes the smallest integer not less than $u$, we get
\begin{equation*}
\ell := \liminf_{n \rightarrow \infty} \frac{\sum_{k,j=1}^{n} \mathbb{P} (A_{k} \cap A_{j})}{\left[\sum_{k=1}^{n} \mathbb{P}(A_{k}) \right]^{2}} \leqslant \varepsilon^{2p}.
\end{equation*}
Since the latter inequality is true for any $\varepsilon > 1$, we have $\ell \leqslant 1$. According to R\'{e}nyi--Lamperti's lemma (see, for instance, \citealt{Chandra08}), we obtain $\ell = 1$ and $\mathbb{P} (\limsup_{k \rightarrow \infty} A_{k})$ $= 1$ which contradicts \eqref{eq:1.3}. Therefore,
\begin{equation}\label{eq:1.4}
\sum_{k=1}^{\infty} \mathbb{P} \big\{X_{1} > k^{1/p} \big\} < \infty.
\end{equation}
It remains to prove that
\begin{equation}\label{eq:1.5}
\sum_{k=1}^{\infty} \mathbb{P} \big\{X_{k} < -k^{1/p} \big\} < \infty.
\end{equation}
By virtue of $\Delta_{k,j}(x,y) = \mathbb{P} \{X_{k} \leqslant x, X_{j} \leqslant y \} - \mathbb{P} \{X_{k} \leqslant x \} \mathbb{P} \{X_{j} \leqslant y \}$ for all $x,y \in \mathbf{R}$, it follows for any $k \neq j$ and each fixed $\varepsilon > 1$,
\begin{align*}
&\int_{-j^{1/p}}^{-j^{1/p}/\varepsilon} \int_{-k^{1/p}}^{-k^{1/p}/\varepsilon} \Delta_{X_{k},X_{j}}(x,y) \, \mathrm{d}x \mathrm{d}y \\
& \quad = \int_{-j^{1/p}}^{-j^{1/p}/\varepsilon} \int_{-k^{1/p}}^{-k^{1/p}/\varepsilon} \big[\mathbb{P} \{X_{k} \leqslant x, X_{j} \leqslant y \} - \mathbb{P} \{X_{k} \leqslant x \} \mathbb{P} \{X_{j} \leqslant y \} \big] \mathrm{d}x \mathrm{d}y \\
& \quad \geqslant \left(\frac{\varepsilon - 1}{\varepsilon} \right)^{2} k^{1/p} j^{1/p} \mathbb{P} \left\{X_{k} \leqslant -k^{1/p}, X_{j} \leqslant -j^{1/p} \right\} \\
&\qquad - \left(\frac{\varepsilon - 1}{\varepsilon} \right)^{2} k^{1/p} j^{1/p} \mathbb{P} \left\{X_{1} \leqslant -\frac{k^{1/p}}{\varepsilon} \right\} \mathbb{P} \left\{X_{1} \leqslant -\frac{j^{1/p}}{\varepsilon} \right\}
\end{align*}
i.e., by setting $B_{k} := \left\{X_{k} \leqslant -k^{1/p} \right\}$ we have
\begin{align*}
\mathbb{P} (B_{k} \cap B_{j}) &\leqslant \left(\frac{\varepsilon}{\varepsilon - 1} \right)^{2} \frac{1}{k^{1/p} j^{1/p}} \int_{-j^{1/p}}^{-j^{1/p}/\varepsilon} \int_{-k^{1/p}}^{-k^{1/p}/\varepsilon} \Delta_{X_{k},X_{j}}(x,y) \, \mathrm{d}x \mathrm{d}y \\
&\qquad + \mathbb{P} \left\{X_{1} \leqslant -\frac{k^{1/p}}{\varepsilon} \right\} \mathbb{P} \left\{X_{1} \leqslant -\frac{j^{1/p}}{\varepsilon} \right\}.
\end{align*}
Assuming $\sum_{k=1}^{\infty} \mathbb{P} \left\{X_{1} \leqslant -k^{1/p} \right\} = \infty$ and following exactly the same steps as before, one can conclude that
\begin{equation*}
\liminf_{n \rightarrow \infty} \frac{\sum_{k,j=1}^{n} \mathbb{P} (B_{k} \cap B_{j})}{\left[\sum_{k=1}^{n} \mathbb{P}(B_{k}) \right]^{2}} = 1,
\end{equation*}
and whence $\sum_{k=1}^{\infty} \mathbb{P} \big\{X_{1} \leqslant -k^{1/p} \big\} < \infty$ entailing \eqref{eq:1.5}. Thus, \eqref{eq:1.4} and \eqref{eq:1.5} establish the thesis.
\end{proof}

Next, we present a sequence of identically distributed and pairwise PQD random variables which satisfies \eqref{eq:1.2}.

\begin{example}
Let $1 \leqslant p < 2$ and $\{X_{n}, \, n \geqslant 1\}$ a sequence of identically distributed random variables such that, for all $k \neq j$, $(X_{k},X_{j})$ has bivariate distribution,
\begin{equation*}
F_{X_{k},X_{j}}(x,y) = F_{X_{k}}(x) F_{X_{j}}(y) + \theta_{k,j} F_{X_{k}}^{s_{k,j}}(x) F_{X_{j}}^{s_{k,j}}(y) \big[1 - F_{X_{k}}(x) \big]^{r_{k,j}} \big[1 - F_{X_{j}}(y) \big]^{r_{k,j}}
\end{equation*}
where $r_{k,j}, s_{k,j} \geqslant 1$ and $0 \leqslant \theta_{k,j} \leqslant 1$ for all $k,j$. Thus, $\{X_{n}, \, n \geqslant 1 \}$ is a pairwise PQD sequence (see, for instance, \citealt{Lai00}). Considering $X_{1}$ having probability density function $f_{X_{1}}(t) = 2t^{-3} I_{(1,\infty)}(t)$, it follows $\mathbb{E} \, \lvert X_{1} \rvert^{p} = 2/(2 - p)$ and $\mathbb{E} \, X_{1}^{2} = \infty$. After standard computations, we obtain for $u,v>1$
\begin{align*}
G_{X_{k},X_{j}}(u,v) &= \theta_{k,j} \left[\frac{s_{k,j} \Gamma(s_{k,j}) \Gamma(r_{k,j} + 1/2)}{(2r_{k,j} - 1) \Gamma(r_{k,j} + s_{k,j} + 1/2)} - \frac{H\left(-s_{k,j},r_{k,j} - 1/2;r_{k,j} + 1/2;1/u^{2} \right)}{(2r_{k,j} - 1)u^{2r_{k,j} - 1}} \right] \\
&\qquad \cdot \left[\frac{s_{k,j} \Gamma(s_{k,j}) \Gamma(r_{k,j} + 1/2)}{(2r_{k,j} - 1) \Gamma(r_{k,j} + s_{k,j} + 1/2)} - \frac{H\left(-s_{k,j},r_{k,j} - 1/2;r_{k,j} + 1/2;1/v^{2} \right)}{(2r_{k,j} - 1)v^{2r_{k,j} - 1}} \right]
\end{align*}
where $\Gamma(z) := \int_{0}^{\infty} t^{z-1} \mathrm{e}^{-t} \mathrm{d}t$, $\mathrm{Re}(z) > 0$, is the gamma function and $H(a,b;c;z) := \sum_{n=0}^{\infty} \frac{(a)_{n} (b)_{n}}{(c)_{n} n!} z^{n}$, $c \neq 0,-1,-2,\ldots$ denotes the hypergeometric function (here, $(a)_{n} := a(a + 1) \ldots (a + n - 1)$ stands for Pochhammer symbol). For instance, supposing $r_{k,j} = r$, $s_{k,j} = s$, and $\theta_{k,j} = k^{\mu} j^{\nu}$ with $\mu, \nu$ verifying $1/p - 1 < \mu < 2/p - 2 - \nu$ we have $G_{X_{k},X_{j}}(k^{1/p},j^{1/p}) \leqslant C \theta_{k,j}$ with $C$ not depending on $k,j$ yielding
\begin{equation*}
\sum_{1 \leqslant k < j \leqslant \infty} (kj)^{-1/p} G_{X_{k},X_{j}}(k^{1/p},j^{1/p}) \leqslant C \sum_{j=2}^{\infty} j^{\nu - 1/p} \sum_{k=1}^{j-1} k^{\mu - 1/p} \leqslant C \sum_{j=2}^{\infty} j^{\mu + \nu - 2/p + 1} < \infty.
\end{equation*}
Similar examples can be obtained considering $X_{1} \in \mathscr{L}_{p}$ and random pairs $(X_{k},X_{j})$, $k \neq j$ having bivariate distribution
\begin{equation*}
F_{X_{k},X_{j}}(x,y) = F_{X_{k}}(x) F_{X_{j}}(y) + \theta_{k,j} \phi_{k,j}\big[F_{X_{k}}(x) \big] \psi_{k,j}\big[F_{X_{j}}(y) \big]
\end{equation*}
by choosing suitable nonnegative absolutely continuous functions $\phi_{k,j}, \psi_{k,j}$ on $[0,1]$ satisfying, for each $k,j$, $\phi_{k,j}(0) = \phi_{k,j}(1) = \psi_{k,j}(0) = \psi_{k,j}(1) = 0$, and
\begin{equation*}
0 \leqslant \theta_{k,j} \leqslant -\frac{1}{\min\big\{\alpha_{k,j} \delta_{k,j}, \beta_{k,j} \gamma_{k,j} \big\}},
\end{equation*}
where $\alpha_{k,j} := \inf \big\{\phi_{k,j}'(t) \colon t \in A \big\} < 0$, $\beta_{k,j} := \sup \big\{\phi_{k,j}'(t) \colon t \in A \big\} > 0$, $\gamma_{k,j} := \inf \big\{\psi_{k,j}'(t) \colon t \in B \big\} < 0$ and $\delta_{k,j} := \inf \big\{\psi_{k,j}'(t) \colon t \in B \big\} > 0$, with $A = \big\{t \in [0,1]\colon \phi_{k,j}'(t) \; \text{exists} \big\}$ and $B = \big\{t \in [0,1]\colon \psi_{k,j}'(t) \; \text{exists} \big\}$ (see Corollary 2.4 of \citealt{Rodriguez-Lallena04}).
\end{example}

By gathering Theorem 4 of \citealt{Chen19} with above Theorem~\ref{thr:1}, we can announce a Marcinkiewicz--Zygmund strong law of large numbers for (positively) associated and identically distributed random variables.

\begin{corollary}\label{cor:1}
Let $1 \leqslant p < 2$ and $\{X_{n}, \, n \geqslant 1 \}$ be a sequence of identically distributed and (positively) associated random variables satisfying condition \eqref{eq:1.2}. Then, $\mathbb{E} \lvert X_{1} \rvert^{p} < \infty$ if and only if $\sum_{k=1}^{n} (X_{k} - \mathbb{E} \, X_{1})/n^{1/p} \overset{\textnormal{a.s.}}{\longrightarrow} 0$.
\end{corollary}

\begin{proof}
Obviously, \eqref{eq:1.2} implies \eqref{eq:1.1} and one implication is a consequence of Theorem 4 in \citealt{Chen19}. The reciprocal follows from Theorem~\ref{thr:1} and Corollary 4.1.3 of \citealt{Chow97}. The proof is complete.
\end{proof}

Furthermore, we are also able to present the following equivalence result.

\begin{corollary}
Let $\{X_{n}, \, n \geqslant 1 \}$ be a sequence of identically distributed and pairwise PQD random variables satisfying $\sum_{1 \leqslant k < j \leqslant \infty} G_{X_{k},X_{j}}(k,j)/(kj) < \infty$. Then, $\mathbb{E} \lvert X_{1} \rvert < \infty$ if and only if $\sum_{k=1}^{n} (X_{k} - \mathbb{E} \, X_{1})/n \overset{\textnormal{a.s.}}{\longrightarrow} 0$.
\end{corollary}

\begin{proof}
The sufficiency part is assured by Remark 1 of \citealt{Lita18b} (or Theorem 3 of \citealt{Chen19}). The necessity part is a consequence of Theorem~\ref{thr:1} and Corollary 4.1.3 of \citealt{Chow97}.
\end{proof}

\begin{acknowledgements}
This work is a contribution to the Project UIDB/04035/2020, funded by FCT - Funda\c{c}\~{a}o para a Ci\^{e}ncia e a Tecnologia, Portugal.
\end{acknowledgements}

\end{document}